\documentclass[leqno,12pt]{amsart}
\usepackage[letterpaper,margin=1in,footskip=0.25in]{geometry}
\usepackage[utf8]{inputenc}
\usepackage{Baskervaldx}
\usepackage[cal=boondoxupr]{mathalfa}
\usepackage{mathrsfs}
\usepackage{enumerate, setspace, amssymb, amsopn, amsmath, array, pdfsync, url, amsfonts, float}
\usepackage[baskervaldx,bigdelims,vvarbb]{newtxmath}

\usepackage[colorlinks=true, citecolor=RoyalBlue, filecolor=RoyalBlue, linkcolor=RoyalBlue, urlcolor=RoyalBlue, hyperindex]{hyperref}
\usepackage[dvipsnames]{xcolor}
\usepackage{tikz, tikz-cd}
\usepackage[capitalise]{cleveref}

\allowdisplaybreaks

\newtheorem{theorem}{Theorem}[section]
\theoremstyle{definition}

\theoremstyle{definition}

\newtheorem{corollary}[theorem]{Corollary}

\theoremstyle{definition}

\newtheorem{remark}[theorem]{Remark}
\theoremstyle{remark}

\newtheorem{example}[theorem]{Example}

\usepackage{stmaryrd}

\newcommand{\depth}{\operatorname{depth}}

\newcommand{\Frac}{\operatorname{Frac}}

\newcommand{\Z}{\mathbb{Z}}
\newcommand{\F}{\mathbb{F}}

\newcommand{\sk}{\mathcal{k}}
\newcommand{\fm}{\mathfrak{m}}
\newcommand{\fp}{\mathfrak{p}}

\newcommand{\fn}{\mathfrak{n}}

\newcommand{\rperf}{R_{\operatorname{perf}}}
\newcommand{\aperf}{A_{\operatorname{perf}}}
\newcommand{\lcm}{\operatorname{lcm}}

\makeatother

\crefname{Theoremx}{Theorem}{Theorems}

\usepackage{titlesec}		
\titleformat{\section}[block]{\large\scshape\bfseries\filcenter}{\thesection.}{1em}{}
\titleformat{\subsection}[hang]{\large\scshape\bfseries}{\thesubsection}{1em}{}	
\titleformat{\subsubsection}[hang]{\large\scshape\bfseries}{\thesubsubsection}{1em}{}

\usepackage[backend=biber, style=alphabetic, sorting=anyt,maxalphanames=5,maxbibnames=99]{biblatex}
\usepackage{hyperref}
\usepackage{xurl}
\hypersetup{breaklinks=true}

\renewbibmacro{in:}{}

\begin{document}

\title{The perfection can be a non-coherent GCD domain}

\author[Austyn Simpson]{Austyn Simpson}
\thanks{Simpson was supported by NSF postdoctoral fellowship DMS \#2202890.}
\address{Department of Mathematics, University of Michigan, Ann Arbor, MI 48109 USA}
\email{austyn@umich.edu}

\begin{abstract}
We show that there exists a complete local Noetherian normal domain of prime characteristic whose perfection is a non-coherent GCD domain, answering a question of Patankar in the negative concerning characterizations of $F$-coherent rings. This recovers and extends a result of Glaz using tight closure methods.
\end{abstract}
\maketitle

\section{Introduction}\label{sec:intro}
Let $R$ be a Noetherian ring of prime characteristic $p>0$. The \emph{perfection} (or \emph{perfect closure}) of $R$, denoted in this note by $\rperf$, is defined as
\begin{align*}
    \rperf&=\lim\limits_{\longrightarrow}\left(R\stackrel{F}{\to} R\stackrel{F}{\to}\cdots\right)
\end{align*}
where $F:R\rightarrow R$ is the Frobenius map $r\mapsto r^p$. If $R$ is further assumed to be a domain with fraction field $K=\Frac(R)$, then the \emph{absolute integral closure} of $R$ is defined to be the integral closure of $R$ in a choice of algebraic closure $\overline{K}$. These two objects $\rperf\subseteq R^+$ have over the past few decades seen substantial use in subjects where one usually only a priori considers Noetherian rings, despite being highly non-Noetherian themselves. It is natural to ask which ``Noetherian-like" properties these objects enjoy. Two such notions of interest in this note are \emph{coherence} (that is, all finitely generated ideals are finitely presented) and the property of being a \emph{GCD domain} (that is, a domain in which the intersection of two principal ideals is principal).

In the case of $R^+$, this can be effectively hopeless due to the fact that $R^+$ usually fails to be coherent regardless of characteristic \cite{AH97,Asg17,Pat22}. For $\rperf$ however, the situation is not as dire: the perfect closure of a regular local ring of prime characteristic is coherent \cite[Proposition 3.3(1)]{Shi11} and its finitely generated ideals enjoy finite primary decomposition \cite[Corollaire 1]{Rad83}, essentially both because of Kunz's theorem \cite{Kun69}. The study of Noetherian rings with coherent perfect closures was pioneered by Shimomoto in \cite{Shi11} who called them \emph{$F$-coherent}, suggesting connections to other classes of singularities defined via the Frobenius map. The results of \cite{Shi11} for $F$-coherent rings $R$ may be summarized as follows, where for simplicity we assume that $R$ is reduced and that $F_*R$ is a finite $R$-module.
\begin{enumerate}
    \item $F$-coherence is stable under localization and descends under faithfully flat maps.
    \item $F$-coherent rings have purely inseparable normalizations.
    \item Regular rings, purely inseparable extensions of regular rings, and purely inseparable subrings of regular rings, are all $F$-coherent.\label{summary-3}
    \item If $(R,\fm)$ is local and $F$-coherent, then $\rperf$ is a big Cohen-Macaulay algebra.\label{summary-4}
    \item Tight closure coincides with Frobenius closure in the $F$-coherent setting. In particular, tight closure localizes, weak $F$-regularity and $F$-purity are equivalent, and $F$-injectivity and $F$-rationality are equivalent.\label{summary-5}
\end{enumerate}
At present, the major shortcoming of this theory is the scarcity of examples. In fact, the only known examples of $F$-coherent rings are those which come from a purely inseparable extension or inclusion of a regular ring as in (\ref{summary-3}). 

This subject has been further explored in \cite{Asg17} using homological methods. For example, it is observed in \cite{Asg17} (relying heavily on ideas introduced in \cite{AH97}) that if $\rperf$ is coherent then it is a GCD domain. The relationship between GCD domains and the coherence property has a peculiar history; for instance, it had been wondered for some time in the broader context of non-Noetherian ring theory whether there exists a non-coherent GCD domain. Such a ring was first constructed in \cite{Gla01} (see Section 1 of \emph{op. cit.} for further context) and a different construction via ultraproducts was provided in \cite{OS03}. The question of whether this phenomenon may occur for the perfect closure of a Noetherian ring (and in particular whether $\rperf$ being a GCD domain characterizes $F$-coherent rings) is posed in \cite[Remark 4.9]{Pat22}. The content of this note is the following:

\begin{theorem}\label{theorem:main-theorem} (= \Cref{cor:gcd-fpure} + \Cref{ex:fpure-ufd-non-CM,ex:fpure-ufd-non-fregular})
There exists a complete local Noetherian normal domain $R$ of prime characteristic $p>0$ such that $\rperf$ is a non-coherent GCD domain.
\end{theorem}

To find such rings, we first observe that if $R$ is a prime characteristic Noetherian UFD, then $\rperf$ is a GCD domain (\Cref{thm:ufd-gcd}). We then appeal to (\ref{summary-5}) above to conclude that any $F$-pure UFD which is not Cohen--Macaulay (or even just not weakly $F$-regular) will satisfy this requirement. Such rings exist by work of Bertin \cite{Ber67} and Fossum-Griffith \cite{FG75}, considering certain rings of invariants in characteristic $2$ (\cref{ex:fpure-ufd-non-CM}). We also provide Cohen--Macaulay examples by considering certain diagonal hypersurfaces (\cref{ex:fpure-ufd-non-fregular}).

\subsection*{Acknowledgments}
I am grateful to Anurag Singh for helpful discussions. I thank Bruce Olberding, Shravan Patankar, Thomas Polstra, Kazuma Shimomoto, and the anonymous referee for helpful suggestions on a previous draft of this note.

\section{Results}\label{sec:results}
Recall that a domain $A$ (not necessarily Noetherian) is a \emph{GCD domain} if for any $r,s\in A$, the ideal $rA\cap sA$ is principal (equivalently, the ideal $(r:_A s)$ is principal). Asgharzadeh has observed using homological methods that if $R$ is a local Noetherian $F$-coherent domain, then $\rperf$ is a GCD domain \cite[Corollary 3.7]{Asg17}. In this section we construct Noetherian rings $R$ which are not $F$-coherent but whose perfections are GCD domains. We begin with an elementary proof that if $R$ is a prime characteristic UFD, then $\rperf$ is a GCD domain irrespective of any coherence assumptions.

\begin{theorem}\label{thm:ufd-gcd}
Let $R$ be a Noetherian UFD of prime characteristic $p>0$. Then $\rperf$ is a GCD domain.
\end{theorem}
\begin{proof}
    Let $a,b\in \rperf$. We will show that $(a:_{\rperf} b)$ is a principal ideal. There exists $N_1\gg 0$ such that $a^{p^{N_1}},b^{p^{N_1}}\in R$. Let
    \begin{align*}
        a^{p^{N_1}}=u\pi_1^{s_1}\cdots \pi_m^{s_m}\\
        b^{p^{N_1}}=v\pi_1^{t_1}\cdots\pi_m^{t_m}
    \end{align*}
be prime factorizations of $a^{p^{N_1}}$ and $b^{p^{N_1}}$ where $u,v\in R$ are units, $\pi_i\in R$ are distinct irreducible elements, and $s_i,t_i\geq 0$. Note that 
\begin{align*}
    \left(a^{p^{N_1}}:_R b^{p^{N_1}}\right)=\left(\frac{\lcm\left(a^{p^{N_1}},b^{p^{N_1}}\right)}{b^{p^{N_1}}}\right)=(c)
\end{align*}
where $c:=\pi_1^{\max\{0,s_1-t_1\}}\cdots \pi_m^{\max\{0,s_m-t_m\}}$. We claim that $(a:_{\rperf} b)=(c^{1/p^{N_1}})\rperf$, where the $\supseteq$ direction is clear. Let $\eta\in(a:_{\rperf} b)$ and write $\eta b=\xi a$ for some $\xi\in\rperf$. Choose $N_2\gg 0$ such that $\eta^{p^{N_1+N_2}},\xi^{p^{N_1+N_2}}\in R$. We see that
\begin{align}
\eta^{p^{N_1+N_2}}\in \left(a^{p^{N_1+N_2}}:_R b^{p^{N_1+N_2}}\right)=\left(\pi_1^{\max\{0,p^{N_2}(s_1-t_1)\}}\cdots \pi_m^{\max\{0,p^{N_2}(s_m-t_m)\}}\right)=\left(c^{p^{N_2}}\right).\label{eq:colon}
\end{align}
Write $\eta^{p^{N_1+N_2}}=rc^{p^{N_2}}$ for some $r\in R$. Taking $p^{N_1+N_2}$-th roots, we have the equation $\eta=r^{1/p^{N_1+N_2}} c^{1/p^{N_1}}$ (in $\rperf$), so $\eta\in(c^{1/p^{N_1}})\rperf$ as desired.
\end{proof}
\begin{remark}
    If $R$ is not a UFD then the equality $(f:_R g)^{[p^e]}=(f^{p^e}:_R g^{p^e})$ in \Cref{eq:colon} need not hold; in fact, such behavior characterizes local prime characteristic UFDs by \cite[Theorem 3.5]{Zha09}.
\end{remark}
We briefly recall for the reader the notions of $F$-purity and weak $F$-regularity. Let $R$ be a Noetherian ring of prime characteristic $p>0$. $R$ is said to be \emph{$F$-pure} if the Frobenius map $F:R\rightarrow F_* R$ is a pure $R$-module homomorphism. If $F_* R$ is a finite $R$-module or if $R$ is a complete local ring, then $R$ is $F$-pure if and only if $R\rightarrow F_* R$ splits as a map of $R$-modules. $R$ is said to be \emph{cyclically $F$-pure} if $I=I^F$ for all ideals $I\subseteq R$, where $I^F:=I\rperf\cap R$ denotes the \emph{Frobenius closure of $I$}. $F$-purity always implies cyclic $F$-purity, and the converse holds, for example, for rings which are excellent and reduced \cite{Hoc77}. Finally, $R$ is said to be \emph{weakly $F$-regular} if $I=I^*$ for all ideals $I\subseteq R$, where $$I^*=\left\{r\in R\mid \text{there exists }c\in R\setminus\bigcup\limits_{\fp\in\min(R)} \fp\text{ such that } cr^{p^e}\in I^{[p^e]}\text{ for all } e\gg 0\right\}$$ denotes the \emph{tight closure of $I$}. We recall the following theorem of Shimomoto since it is a key ingredient in our construction.

\begin{theorem}\cite[Proposition 3.12 \& Corollary 3.15]{Shi11}\label{theorem:f-pure-f-regular}
    Let $R$ be a reduced cyclically $F$-pure local ring of prime characteristic $p>0$ whose perfection $\rperf$ is coherent. Then $R$ is weakly $F$-regular. If $R$ is further assumed to be excellent, then $R$ is Cohen--Macaulay. 
\end{theorem}

The proof of \cref{theorem:f-pure-f-regular} uses a valuation argument to show that coherence of $\rperf$ implies the equality $I^F=I^*$ for all ideals $I\subseteq R$, and weakly $F$-regular excellent rings are well-known to be Cohen--Macaulay. Despite having a simple proof, \Cref{thm:ufd-gcd} has the following unexpected consequence when combined with \cref{theorem:f-pure-f-regular}.

\begin{corollary}\label{cor:gcd-fpure}
Let $R$ be a local Noetherian $F$-pure UFD of prime characteristic $p>0$. Then $\rperf$ is a GCD domain which is not coherent in either of the following two cases:
\begin{enumerate}
    \item $R$ is not weakly $F$-regular;\label{cor:gcd-fpure-1}
    \item $R$ is excellent and not Cohen--Macaulay.\label{cor:gcd-fpure-2}
\end{enumerate}
\end{corollary}
\begin{proof}
    $\rperf$ is a GCD domain by \cref{thm:ufd-gcd}. Since $R$ is $F$-pure, it is reduced and cyclically $F$-pure. The fact that $\rperf$ is not coherent in the two cases above now follows from the contrapositives of both statements in \cref{theorem:f-pure-f-regular}, as desired.
\end{proof}
We demonstrate that rings satisfying the above hypotheses exist:
\begin{example}\label{ex:fpure-ufd-non-CM}
Let $\sk$ be a field of characteristic $2$ and let $B:=\sk[x_0,x_1,x_2,x_3]$. Consider the action of the group $G:=\Z/4\Z$ on $B$ induced by the $\sk$-algebra automorphism $\sigma(x_i)=x_{i+1}$ (where the indices are viewed $\mod 4$).  Let $\fm=(x_0,x_1,x_2,x_3)$ be the homogeneous maximal ideal of $B$, and let $A=B^G$ be the invariant subring. It is shown in \cite{Ber67} that $A$ is a UFD which is not Cohen--Macaulay. Moreover, $A$ is $F$-pure by \cite[Proposition 2.4(a)]{Gla95}. Let $\fn=\fm\cap A$, and denote $R=\widehat{A_\fn}$. $R$ is $F$-pure by \cite[Exercise 11 \& Corollary 2.3]{MP21}, and it is shown in \cite{FG75} that $R$ is a UFD. It follows that $\rperf$ is a non-coherent GCD domain by \cref{cor:gcd-fpure} (\ref{cor:gcd-fpure-2}). \hfill \qed
\end{example}

We can also use (\ref{cor:gcd-fpure-1}) instead of (\ref{cor:gcd-fpure-2}) in \cref{cor:gcd-fpure} to produce Cohen--Macaulay examples exhibiting the conclusion of \cref{theorem:main-theorem}.
\begin{example}\label{ex:fpure-ufd-non-fregular}
Let $p\equiv 1\mod 5$ and consider the degree $5$ diagonal hypersurface $$A=\F_p[x,y,z,u,v]/(x^5+y^5+z^5+u^5+v^5)$$ with homogeneous maximal ideal $\fm$. By the Jacobian criterion, $\fm$ defines the singular locus of $A$, hence $A$ and its $\fm$-adic completion $R:=\widehat{A}$ are UFDs by Grothendieck's parafactoriality theorem \cite[XI Corollaire 3.10 and XI Th\'{e}or\`{e}me 3.13(ii)]{SGA2} (see also \cite{CL94}). By the assumption on the characteristic, the coefficient of $(xyzuv)^{p-1}$ in the monomial expansion of $(x^5+y^5+z^5+u^5+v^5)^{p-1}$ is nonzero, hence $A$ and $R$ are both $F$-pure by Fedder's criterion \cite{Fed83}. However, it is well-known that $A$ and $R$ are not weakly $F$-regular. Indeed, the ideal $(y,z,u,v)$ is not tightly closed in either ring as one checks that $x^4\in (y,z,u,v)^*$ (see also \cite[Exercise 17]{MP21} and \cite[Example 1.6.3]{Hun96}). We then apply \cref{cor:gcd-fpure} (\ref{cor:gcd-fpure-1}) to conclude that $\rperf$ is a non-coherent GCD domain.

We remark that this reasoning applies more generally to any $F$-pure local hypersurface of dimension at least four which is not weakly $F$-regular and which has an isolated singularity.\hfill \qed
\end{example}

We may also deduce non-explicitly that there are plenty of the eponymous rings by combining \cref{cor:gcd-fpure} with a well-known result of Heitmann, as suggested to the author by B. Olberding:

\begin{corollary}\label{cor:precompletion}
    Let $(T,\fm)$ be a complete local domain with $\depth T\geq 2$ which is $F$-pure but not weakly $F$-regular. Then there exists an excellent local ring $(A,\fm\cap A)$ such that $\aperf$ is a non-coherent GCD domain and such that $\widehat{A}\cong T$.
\end{corollary}
\begin{proof}
    We use \cite[Theorem 8]{Hei93} to construct an excellent local UFD $A$ such that $\widehat{A}\cong T$. Since $T$ is not weakly $F$-regular, the excellence of $A$ together with \cite[Corollary 7.28]{HH94} implies that $A$ is not weakly $F$-regular. Since $A$ is $F$-pure, we may apply \cref{cor:gcd-fpure} (\ref{cor:gcd-fpure-1}) to conclude that $\aperf$ is a non-coherent GCD domain.
\end{proof}

We conclude with the following remarks.

\begin{remark}
\begin{enumerate}
    \item \cref{cor:gcd-fpure} (\ref{cor:gcd-fpure-2}) does not provide additional content to \cref{cor:precompletion} because all rings under consideration are excellent, and excellent weakly $F$-regular rings are Cohen--Macaulay by \cite[Proposition 4.2(c)]{HH94}.
    \item In view of \Cref{thm:ufd-gcd} we may ask whether the property of $\rperf$ being a GCD domain characterizes those rings $R$ sharing a perfect closure with a prime characteristic UFD.
    \item \cref{cor:precompletion} suggests a potential approach to showing that the completion of a local $F$-coherent ring need not be $F$-coherent. The existence of such a ring is unknown to the author.
\end{enumerate}
     \end{remark}

\printbibliography

\end{document}